\documentclass{article}
\usepackage{amsmath,amsthm,amssymb,amsfonts}
\usepackage{tikz}
\usepackage{xcolor}
\usepackage{amsrefs}

\newcommand{\groundmodel}{\mathcal{N}}
\newcommand{\refl}{\mathsf{Ref}}
\newcommand{\bew}[2]{\mathrm{Pr}_{ #1 }(\ulcorner #2 \urcorner)}
\newcommand{\sigmaoneisdc}{\mathrm{Strong} \; \Sigma^1_i \text{-}\mathsf{DC}_0}
\newcommand{\tuple}[1]{{\langle #1 \rangle}}

\theoremstyle{plain}
\newtheorem{theorem}{Theorem}
\newtheorem{lemma}[theorem]{Lemma}
\newtheorem{corollary}[theorem]{Corollary}
\newtheorem{proposition}[theorem]{Proposition}

\newtheorem*{question}{Question}
\theoremstyle{definition}
\newtheorem{definition}[theorem]{Definition}

\title{Determinacy and reflection principles in second-order arithmetic}
\author{Leonardo Pacheco$^1$, Keita Yokoyama$^2$\\
        $^{1,2}$Mathematical Institute, Tohoku University, Sendai, Japan\\
        $^1$\texttt{leonardovpacheco@gmail.com}, $^2$\texttt{keita.yokoyama.c2@tohoku.ac.jp}}
\date{}

\begin{document}
\maketitle

\begin{abstract}
    It is known that several variations of the axiom of determinacy play important roles in the study of reverse mathematics, and the relation between the hierarchy of determinacy and comprehension are revealed by Tanaka, Nemoto, Montalb{\'a}n, Shore, and others.
    We prove variations of a result by Ko{\l }odziejczyk and Michalewski relating determinacy of arbitrary boolean combinations of $\Sigma^0_2$ sets and reflection in second-order arithmetic.
    Specifically, we prove that:
      over $\mathsf{ACA}_0$,
        $\Pi^1_2$-$\refl(\mathsf{ACA}_0)$ is equivalent to $\forall n.(\Sigma^0_1)_n$-$\mathsf{Det}^*$;
        $\Pi^1_3$-$\refl(\Pi^1_1$-$\mathsf{CA}_0)$ is equivalent to $\forall n.(\Sigma^0_1)_n$-$\mathsf{Det}$; and
        $\Pi^1_3$-$\refl(\Pi^1_2$-$\mathsf{CA}_0)$ is equivalent to $\forall n.(\Sigma^0_2)_n$-$\mathsf{Det}$.
    We also restate results by Montalb{\'a}n and Shore to show that $\Pi^1_3$-$\refl(\mathsf{Z}_2)$ is equivalent to $\forall n.(\Sigma^0_3)_n$-$\mathsf{Det}$ over $\mathsf{ACA}_0$.

\end{abstract}

In this paper we characterize several determinacy axioms by reflection principles in second-order arithmetic.
Consider axioms stating the determinacy of Gale--Stewart games whose payoff sets are finite (possibly nonstandard) differences of $\Sigma^0_1$, $\Sigma^0_2$ and $\Sigma^0_3$ sets in Baire or Cantor space.
We characterize these axioms by reflection principles $\Pi^1_n$-$\refl(T)$ stating that ``every $\Pi^1_n$-formula provable in $T$ is true''.
We show that
\begin{theorem}\label{thm::main_theorem}
    For $k\in\mathbb{N}$, let $(\Sigma^0_n)_k$ be the $k^\text{th}$ level of the difference hierarchy of $\Sigma^0_n$.
    Denote by $(\Sigma^0_n)_k$-$\mathsf{Det}$ the formula stating the determinacy of subsets of the Baire space in $(\Sigma^0_n)_k$, and by $(\Sigma^0_n)_k$-$\mathsf{Det}^*$ the formula stating the determinacy of subsets of the Cantor space in $(\Sigma^0_n)_k$.
    Over $\mathsf{ACA}_0$:
    \begin{enumerate}
        \item $\Pi^1_2$-$\refl(\mathsf{ACA}_0)$ is equivalent to $\forall n.(\Sigma^0_1)_n$-$\mathsf{Det}^*$;
        \item $\Pi^1_3$-$\refl(\Pi^1_1$-$\mathsf{CA}_0)$ is equivalent to $\forall n.(\Sigma^0_1)_n$-$\mathsf{Det}$;
        \item $\Pi^1_3$-$\refl(\Pi^1_2$-$\mathsf{CA}_0)$ is equivalent to $\forall n.(\Sigma^0_2)_n$-$\mathsf{Det}$; and
        \item $\Pi^1_3$-$\refl(\mathsf{Z}_2)$ is equivalent to $\forall n.(\Sigma^0_3)_n$-$\mathsf{Det}$.
    \end{enumerate}
\end{theorem}

One of the first results of reverse mathematics was the equivalence of $\Sigma^0_1$-$\mathsf{Det}$ and $\mathsf{ATR}$, proved by Steel \cite{steel1977phdthesis}.
Tanaka \cite{tanaka1990weak} proved that $(\Sigma^0_1)_n$-$\mathsf{Det}$ is equivalent to $\Pi^1_1$-$\mathsf{CA}_0$ over $\mathsf{ACA}_0$, for all $2\leq n<\omega$; Tanaka also proved the equivalence between $\Sigma^0_1\text{-}\mathsf{Det}$ and $\mathsf{ATR}_0$ without using transfinite $\Sigma^1_1$-induction.
Nemoto \emph{et al.} \cite{nemoto2007infinite} proved that $(\Sigma^0_1)_n$-$\mathsf{Det}^*$ is equivalent to $\mathsf{ACA}_0$ over $\mathsf{RCA}_0$, for all $2\leq n<\omega$.
MedSalem and Tanaka studied the determinacy of differences of $\Sigma^0_2$ sets in \cites{medsalem2007delta03,medsalem2008weak}.
Montalb{\' a}n and Shore \cite{montalban2012limits} proved that $\mathsf{Z}_2$ proves $(\Sigma^0_3)_n$-$\mathsf{Det}$ for all $1\leq n<\omega$, and that the converse doesn't hold.
They also showed in \cite{montalban2014limits} that $(\Sigma^0_3)_n$-$\mathsf{Det}$ proves the existence of $\beta$-models of $\Delta^1_{n+1}$-$\mathsf{CA}_0$ for $1\leq n< \omega$.
Ko{\l}odziejczyk and Michalewski \cite{kolodziejczyk2016unprovable} proved item (3) of the theorem above with $\Pi^1_2$-$\mathsf{CA}_0$ as the base system.
Their proof uses a result from Heinatsch and M{\"o}llerfeld \cite{heinatsch2010determinacy} characterizing the $\mu$-calculus in second-order arithmetic via proof-theoretic methods;
while our proof uses work of MedSalem and Tanaka \cite{medsalem2008weak} on the reverse mathematics of inductive definitions.
See \cite{yoshii2017survey} for further results on the reverse mathematics of determinacy.

In this paragraph and the next, we briefly survey existing results on reflection and second-order arithmetic.
Our reflection principles $\Pi^1_m$-$\refl(\Gamma)$ are equivalent to uniform reflection schemes $\Pi^1_n$-$\mathsf{REF}(\Gamma)$.
These schemes have been thoroughly studied in the setting of first-order arithmetic \cites{kreisel1968reflection,leivant1983optimality,beklemishev2005reflection}.
The characterization of reflection schemes was extended to second-order arithmetic by Frittaion \cite{frittaion2022uniformreflection}.
He showed that if $T_0$ is a theory extending $\mathsf{ACA}_0$ and axiomatizable by a $\Pi^1_{k+2}$ sentence, then $T_0 + \Pi^1_{n+2}\refl(T) = T_0 + \Pi^1_n\text{-}\mathsf{TI}(\varepsilon_0)$; where $T$ is $T_0$ plus full induction and $\Pi^1_n\text{-}\mathsf{TI}(\varepsilon_0)$ is transfinite induction up to $\varepsilon_0$ for $\Pi^1_n$-formulas.

Other forms of reflection have also been studied in the literature.
Pakhomov and Walsh \cites{pakhomov2018reflection,pakhomov2021reducing,pakhomov2021infinitary} studied iterated reflection principles and their proof theoretic ordinals.
Pakhomov and Walsh \cite{pakhomov2021reducing} also related $\omega$-model reflection and iterated syntactic reflection.
Arai \cite{arai1998someresults} characterized $\Pi^1_n$-transfinite induction as a reflection principle for $\omega$-logic; Cord{\'o}n-Franco \cite{cordon2017predicativity} \emph{et al.} did the same for $\mathsf{ATR}_0$,  and Fern{\' a}ndez-Duque \cite{fernandez2022omegalogic} for $\Pi^1_1$-$\mathsf{CA}_0$.
For more results on reflection see Simpson's book \cite{simpson2009sosoa} or the first author's survey \cite{pacheco2021rims}.

In section \ref{section::preliminaries}, we define the systems used in the paper, review existing results on determinacy, and prove that our reflection principles are equivalent to uniform reflection schemes.
We also prove some folklore results on determinacy.
In section \ref{section::sequences_of_beta_models} we prove the main lemma of the paper: the reflection principle $\Pi^1_{e+2}$-$\refl(\sigmaoneisdc)$ is equivalent to the existence of arbitrarily long chains of coded $\beta_i$-models which are coded $\beta_e$-submodels of the ground model.
We also use this lemma to prove (1) and (2) of Theorem \ref{thm::main_theorem}.
In section \ref{section::the_pi12-ca_case}, we prove the existence of sequences of $\beta_2$-models using multiple $\Sigma^1_1$-inductions; we then prove item (3) of Theorem \ref{thm::main_theorem}.
In section \ref{section::the_z2_case} we show item (4) of Theorem \ref{thm::main_theorem} using results from Montalb{\'a}n and Shore \cites{montalban2012limits,montalban2014limits}.

\textbf{Acknowledgements.} We would like to thank Leszek Ko{\l}odziejczyk and Yudai Suzuki for their comments on this paper.
The work of the second author is partially supported by JSPS KAKENHI grant numbers 19K03601 and 21KK0045.
We would also like to thank the support from the JSPS--RFBR Bilateral Joint Research Project JSPSBP120204809.

\section{Preliminaries}\label{section::preliminaries}
\subsection{Reverse Mathematics}
In this subsection we review some basic definitions for reverse mathematics in second-order arithmetic.
The standard reference for reverse mathematics is \cite{simpson2009sosoa}.
Reserve $\mathbb{N}$ for the set of natural numbers in second-order arithmetic and $\omega$ for the ``real'' set of natural numbers.

We consider formulas in the language $\mathcal{L}_2$ of second-order arithmetic.
A formula is $\Sigma^0_0$ iff it has only bounded number quantifiers; $\Sigma^0_{k+1}$ if it is equivalent to some formula of the form $\exists x.\varphi$ with $\varphi\in\Pi^0_k$; $\Pi^0_k$ if it is equivalent to the negation of a $\Sigma^0_k$ formula; and $\Delta^1_0$ if it is $\Sigma^0_n$ for some $n$.
Put $\Sigma^1_0 =\Delta^1_0$.
A formula is $\Sigma^1_{k+1}$ iff it is equivalent to a formula of the form $\exists X.\varphi$, with $\varphi\in \Pi^1_k$ for some $k\in\mathbb{N}$; $\Pi^1_k$ if it is equivalent to the negation of some $\Sigma^1_k$-formula.
A formula is $\Sigma^{i,X}_j$ iff its only set parameter is $X$.

Fix $i\in \{0,1\}$ and $k\in \omega$.
Let $\pi^i_k(e,\bar m, \bar X)$ be a $\Pi^i_k$-formula with exactly $e,\bar m$ as free number variables and $\bar X$ as free set variables.
$\pi^i_k$ is a universal lightface $\Pi^i_k$-formula iff for each $\Pi^i_k$ formula $\pi'$ with the same free variables as $\pi^i_k$, $\mathsf{RCA}_0$ proves
\[
    \forall e \exists e' \forall \bar m \forall \bar X.\pi^i_k(e',\bar m, \bar X) \leftrightarrow \pi'(e,\bar m, \bar X).
\]

We list the set-existence axiom systems used in this paper:
\begin{itemize}
    \item $\mathsf{RCA}_0$ contains the axioms for discrete ordered semirings, $\Sigma^0_1$-induction, and $\Delta^0_1$-comprehension.
    \item $\mathsf{ACA}_0$ is obtained by adding arithmetical comprehension to $\mathsf{RCA}_0$.
    We can alternatively characterize $\mathsf{ACA}_0$ by the axiom ``for all $X$, the Turing jump $\mathrm{TJ}(X)$ of $X$ exists''.
    \item $\mathsf{ACA}_0'$ is obtained by adding ``for all $n$ and $X$, there is a sequence $\tuple{X_0, \dots, X_n}$ such that $X_0 = X$ and $X_{k+1} = \mathrm{TJ}(X_k)$ for all $k<n$'' to $\mathsf{RCA}_0$.
    This axiom is read as ``for all $n$ and $X$, the iterated Turing jump $\mathrm{TJ}^n(X)$ of $X$ exists''.
    \item $\Pi^1_n\text{-}\mathsf{CA}_0$ is obtained by adding $\Pi^1_1$-comprehension to $\mathsf{RCA}_0$.
    Alternatively, we can add ``for all $X$, the hyperjump $\mathrm{HJ}(X)$ of $X$ exists'' to $\mathsf{RCA}_0$.
    \item $\Pi^1_1\text{-}\mathsf{CA}_0'$ is obtained by adding ``for all $n$ and $X$, there is a sequence of coded $\beta$-models $\tuple{X_0, \dots, X_n}$ such that $X\in X_0$ and, for all $i<n$, $X_i\in X_{i+1}$ and $X_i\subseteq_\beta X_{i+1}$'' to $\mathsf{RCA}_0$.
    An alternative formulation is to add to $\mathsf{RCA}_0$ the statement ``for all $n$ and $X$, the iterated hyperjump $\mathrm{HJ}^n(X)$ of $X$ exists'' to $\mathsf{RCA}_0$.
    \item $\mathrm{Strong} \; \Sigma^1_k \text{-}\mathsf{DC}_0$ is $\mathsf{ACA}_0$ plus the following scheme:
    \[
        \exists Z \forall n \forall Y.(\eta(n, (Z)^n,Y) \to \eta(n, (Z)^n,(Z)_n)),
    \] where $\eta(n,X,Y)$ is a $\Sigma^1_k$-formula in which $Z$ does not occur, $(Z)^n= \{(i,m) \mid (i,m)\in Z \land m<n\}$, and $(Z)_n = \{i \mid (i,n)\in Z\}$.
    \item $[\Sigma^1_1]^k\text{-}\mathsf{ID}$ states the existence of the sets inductively definable by combinations of $k$-many $\Sigma^1_1$-transfinite induction operators.
\end{itemize}
As far as we know, $\Pi^1_1\text{-}\mathsf{CA}_0'$ does not appear in the literature.
Its naming is in parallel to $\mathsf{ACA}_0'$.

Let $k\in\omega$, then $\mathcal{M}$ is a (coded) $\beta_k$-model iff every $\Pi^1_k$-sentence $\varphi$ with parameters in $\mathcal{M}$ is true in $\mathcal{M}$ iff it is true in the ground model.
For $k\geq 1$, $\mathrm{Strong} \; \Sigma^1_k \text{-}\mathsf{DC}_0$ is equivalent to ``for all $X$, there is a $\beta_k$-model $\mathcal{M}$ such that $X\in\mathcal{M}$''.
If $i = 1,2$, then $\mathrm{Strong} \; \Sigma^1_k \text{-}\mathsf{DC}_0$ is equivalent to $\Pi^1_k\text{-}\mathsf{CA}_0$.
Furthermore, if we assume $V = L$, then $\mathrm{Strong} \; \Sigma^1_k \text{-}\mathsf{DC}_0$ is equivalent to $\Pi^1_k\text{-}\mathsf{CA}_0$ for any $k$.
We denote the ground model by $\groundmodel$.
Given a coded model $\mathcal{M}$ and an $\mathcal{L}_2$-sentence $\varphi$ with parameters in $\mathcal{M}$, write $\mathcal{M}\models\varphi$ to mean that $\mathcal{M}$ satisfies $\varphi$.
See Section VII.2 of \cite{simpson2009sosoa} for the definition of the satisfaction relation $\models$ in second order arithmetic.

Let $\mathcal{M}, \mathcal{N}$ be coded models.
The sets $(\mathcal{M})_i$ of $\mathcal{M}$ are defined by $(\mathcal{M})_i = \{n \in \mathbb{N} \mid \tuple{n,i} \in \mathcal{M}\}$.
$\mathcal{M}$ is a submodel of $\mathcal{N}$ iff every set in $\mathcal{M}$ is also in $\mathcal{N}$, that is, for each $i\in\mathbb{N}$ there is $j\in\mathbb{N}$ such that $(\mathcal{M})_i = (\mathcal{N})_j$.
Given two coded models $\mathcal{M}$ and $\mathcal{N}$, $\mathcal{M}$ is a $\beta_k$-submodel of $\mathcal{N}$ iff for all $\Sigma^1_k$-formula $\varphi$ with parameters in $\mathcal{M}$, $\mathcal{M}\models \varphi\iff \mathcal{N}\models\varphi$.
When $\mathcal{M}$ is a $\beta_k$-submodel of $\mathcal{N}$, we write $\mathcal{M}\subseteq_{\beta_k}\mathcal{N}$.

\subsection{Reflection Principles}
Let $T$ be a finitely axiomatizable theory.
Let $\mathrm{Pr}_T$ be a standard provability predicate for $T$ and $\mathrm{Tr}_{\Pi^1_n}$ be a truth predicate for $\Pi^1_n$-sentences.
(Here, we only consider $\Pi^1_n$-sentences whose arithmetical parts are $\Sigma^{0}_{2}$ or $\Pi^{0}_{2}$ obtained by the normal form theorem. With $\mathsf{ACA}_{0}'$, we may also consider $\Pi^1_n$-sentences with nonstandard length arithmetical parts, but we do not need to use them in our discussion.)
The reflection principle $\Pi^1_n\text{-}\refl(T)$ is the sentence
\[
    \forall \varphi\in\Pi^1_n.\mathrm{Pr}_{T}(\ulcorner \varphi \urcorner) \rightarrow \mathrm{Tr}_{\Pi^1_n}(\ulcorner\varphi\urcorner).
\]
Note that we consider all $\Pi^1_n$-sentences inside our system, including `nonstandard' sentences.
It is also common to consider reflection \emph{schemes}.
Our principle is equivalent to the uniform reflection scheme $\Pi^1_n\text{-}\mathsf{RFN}(T)$, which consists of the formulas
    \[
        \forall x.\mathrm{Pr}_{T}(\ulcorner \varphi(x) \urcorner) \rightarrow \varphi(x)
    \]
for all (standard) $\Pi^1_n$-formulas.
In general,
\begin{proposition}
    Let $T$ be a finitely axiomatizable theory extending $\mathsf{ACA}_0$.
    $\Pi^1_n\text{-}\refl(T)$ and $\Pi^1_n\text{-}\mathsf{RFN}(T)$ are equivalent over $\mathsf{ACA}_0$.
\end{proposition}
\begin{proof}
    Work inside a model of $\mathsf{ACA}_0$.
    First suppose that the reflection principle $\Pi^1_n\text{-}\refl(T)$ holds.
    Fix a standard formula $\varphi(x)$.
    Let $a\in \mathbb{N}$ be such that $\bew{T}{\varphi(a)}$.
    $\varphi(a)$ is a sentence inside the model, so we can use the reflection principle, so $\varphi(a)$ is true.

    Now, suppose the reflection scheme $\Pi^1_n\text{-}\mathsf{RFN}(T)$ holds.
    Let $\varphi$ be a sentence inside the model such that $\bew{T}{\varphi}$ holds.
    For a contradiction, suppose that $\varphi$ is false.
    Fix a standard universal lightface $\Pi^1_n$-formula $\pi^1_n(x)$.
    For some $e\in \mathbb{N}$, $\bew{\mathsf{RCA}_0}{\varphi\leftrightarrow\pi^1_n(e)}$ holds.
    Therefore $\bew{T}{\pi^1_n(e)}$.
    The reflection instance for $\pi^1_n$ implies that $\pi^1_n(e)$ is true.

    Note that we may take the arithmetical parts of $\varphi$ and $\pi^1_n$ to be $\Sigma^0_2$.
    That is, $\neg \varphi$ is $\exists X\forall m_0 \exists m_1. \psi(X,m_0,m_1)$.
    Fix $X_0$ such that $\forall m_0 \exists m_1. \psi(X_0,m_0,m_1)$ holds.
    We can use $\mathsf{ACA}_0$ to construct an $\omega$-model $\mathcal{M}$ which satisfies $\mathsf{RCA}_0$ and includes $X_0$.
    Since $X_0\in\mathcal{M}$, we have $\mathcal{M}\models \mathsf{RCA}_0 + \neg\varphi$, and so $\mathcal{M}\models \neg \pi^1_n(e)$.
    Since the arithmetical part of $\pi^1_n$ is $\Sigma^0_2$, we can show that $\mathcal{M}\models \pi^1_n(e)$ also holds.
    We conclude that $\varphi$ is true.
\end{proof}

\subsection{Determinacy Axioms}
Let $X$ be some set and $A\subseteq X^\mathbb{N}$.
In a Gale--Stewart game, two players $\mathsf{I}$ and $\mathsf{II}$ alternate picking elements of $X$.
The player $\mathsf{I}$ wins a run $\alpha$ of the game iff $\alpha\in A$.
A Gale--Stewart game is determined iff there is a winning strategy for one of the players.
We consider here Gale--Stewart games on $\mathbb{N}$ and $\{0,1\}$.
The axiom $\Gamma$-$\mathsf{Det}$ states that every game on $\mathbb{N}$ whose payoff is $\Gamma$-definable is determined; and $\Gamma$-$\mathsf{Det}^*$ states the same for games on $\{0,1\}$.
For an introduction to determinacy in second-order arithmetic, see Sections V.8 and VI.5 of \cite{simpson2009sosoa}.
Yoshii \cite{yoshii2017survey} surveys results on the reverse mathematics of determinacy.

The difference hierarchy of $\Sigma^0_k$ captures all boolean combinations of $\Sigma^0_k$ sets.
It is usually defined by induction: a set $X$ is $(\Sigma^0_k)_{1}$ iff $X$ is $\Sigma^0_k$; and $X$ is $(\Sigma^0_k)_{n+1}$ iff $X$ is the difference of a $\Sigma^0_k$ set and a $(\Sigma^0_k)_n$ set.
To state the determinacy axioms studied in this paper we need a formalized version of the difference hierarchy.
Fix $k\in\omega$.
Let $x$ be a number variable and $f$ be a distinguished second-order variable.
$\varphi(f)$ is a $(\Sigma^0_k)_x$ formula iff there is a $\Sigma^0_k$ formula $\psi(y,f)$ (possibly with other free variables) such that: $\psi(x,f)$ always holds; if $z<y<x$ then $\psi(z,f)$ implies $\psi(y,f)$; and $\varphi(f)$ holds iff the least $y\leq z$ such that $x=y\lor \psi(y,f)$ holds is even.
Intuitively, we think of $\varphi(f)$ as the disjunction $\psi(0,f) \lor \psi(2,f)\lor \cdots \lor \psi(2 \lfloor{x/2}\rfloor,f)$.

Given $k\in\omega$, $\forall n.(\Sigma^0_k)_n\text{-}\mathsf{Det}$ is the statement
\[
    \forall n (\exists \sigma \forall \tau \varphi(\sigma\otimes\tau) \lor \exists \tau \forall \sigma \neg\varphi(\sigma\otimes\tau)),
\]
where $\varphi$ is a $(\Sigma^0_k)_n$-formula.
Here, $\sigma\otimes\tau$ is the play obtained when $\mathsf{I}$ uses the strategy $\sigma$ and $\mathsf{II}$ uses the strategy $\tau$.
The definition for $\forall n.(\Sigma^0_k)_n\text{-}\mathsf{Det}^*$ is obtained by restricting the players to playing in the Cantor space.

We now prove a few folklore results.
As far as we are aware, there results are not published, but their proofs are similar to other existing proofs.
We first look at determinacy on Cantor space: Nemoto \emph{et al.} \cite{nemoto2007infinite} proved that the $\Delta^0_1$-$\mathsf{Det}^*$ and $\Sigma^0_1$-$\mathsf{Det}^*$ are equivalent to $\mathsf{WKL}_0$, and $(\Sigma^0_1)_2$-$\mathsf{Det}^*$ is equivalent to $\mathsf{ACA}_0$.
Furthermore, $(\Sigma^0_1)_n\text{-}\mathsf{Det}^*$ is equivalent to $\mathsf{ACA}_0$ for all $n\geq 2$.
On the other hand, determinacy of arbitrary finite differences of open sets is equivalent to $\mathsf{ACA}_0'$.
\begin{proposition}\label{prop::sigma01det-n_implies_ACA0prime}
    Over $\mathsf{ACA}_0$, $\forall n.(\Sigma^0_1)_n\text{-}\mathsf{Det}^*$ implies $\mathsf{ACA}_0'$.
\end{proposition}
\begin{proof}
    Fix $n\in\mathbb{N}$ and $X\subseteq \mathbb{N}$.
    We prove the existence of a sequence $\tuple{X_0, \dots, X_n}$ of sets such that $X_0 = X$ and $X_{i+1} = \mathrm{TJ}(X_i)$, for $i<n$.
    Let $\pi^0_1$ be a fixed universal lightface $\Pi^0_1$-formula.
    The Turing Jump $\mathrm{TJ}(X)$ of $X$ is the set of $m$ such that $\neg \pi^0_1(m,X)$ holds.

    We claim that there is a set $\tilde{\tau}$ computing universal winning strategies for all $(\Sigma^0_1)_{k}$ games on Cantor space with only $X$ as a parameter, for all $k\in\mathbb{N}$.
    Consider the following game: $\mathsf{I}$ chooses an index $e$ for the payoff of a $(\Sigma^0_1)_{k}$ game, coded as a sequence of $e$-many $0$s followed by a $1$; $\mathsf{II}$ answers with their choice of role in the game with index $e$; then $\mathsf{I}$ and $\mathsf{II}$ play the game with index $e$; whoever wins the subgame wins the whole game.
    The payoff of this game has complexity $(\Sigma^0_1)_{2k+2}$.
    As we have $\forall n.(\Sigma^0_1)_n\text{-}\mathsf{Det}^*$, the game above is determined.
    Since $\mathsf{I}$ cannot win, $\mathsf{II}$ has a winning strategy $\tilde{\tau}$; this $\tilde{\tau}$ computes winning strategies for all $(\Sigma^0_1)_{k}$ games.

    Since we have $\forall n.(\Sigma^0_1)_n\text{-}\mathsf{Det}^*$, we also have $(\Sigma^0_1)_2\text{-}\mathsf{Det}^*$, and so $\mathsf{ACA}_0$ is true.
    Let $\mathcal{M}$ be a strict $\beta$-model including $X$ and a universal winning strategy $\tilde{\tau}$ for $(\Sigma^0_1)_{2^{n+2}}$.
    The strict $\beta$-model $\mathcal{M}$ satisfies $\mathsf{WKL}_0$ and induction.
    Since $\tilde{\tau}\in\mathcal{M}$, every $(\Sigma^0_1)_{2^{n+2}}$ game on Cantor space with only $X$ as a parameter is determined in $\mathcal{M}$.

    We consider the following two player game inside $\mathcal{M}$.
    $\mathsf{I}$ starts the game by playing $\tuple{m,i}$, with $i\leq n$, to ask $\mathsf{II}$ whether $m\in X_i$.
    $\mathsf{II}$ then plays $1$ to answer `yes', or $0$ to answer `no'.
    If $\mathsf{II}$ answers $1$, they must show that $m\in X_i$; if $\mathsf{II}$ answers $0$, $\mathsf{I}$ must show that $m\in X_i$.
    Denote by $\mathsf{V}$ (verifier) the player who is trying to show $m\in X_i$ and by $\mathsf{R}$ (refuter) the other player.

    Suppose $i>0$.
    $\mathsf{V}$ wants to show $m\in X_i$.
    So they play a finite sequence $X^f_{i-1}$ witnessing so.
    $X^f_{i-1}$ is intended to be an initial segment of $X_{i-1}$.
    Then, $\mathsf{R}$ plays $m'\leq \mathrm{lh}(X^f_{i-1})$, to state that $X_{i-1}(m') \neq X^f_{i-1}(m')$.
    If $X^f_{i-1}(m') = 0$, $\mathsf{R}$ is stating that $m\in X_{i-1}$, so $\mathsf{V}$ and $\mathsf{R}$ exchange roles.
    Otherwise, $\mathsf{R}$ is asking $\mathsf{V}$ to show that $m\in X_{i-1}$, and the roles stay the same.
    Either way, $\mathsf{V}$ and $\mathsf{R}$ proceed to argue whether $m'\in X_{i-1}$.

    Now suppose $\mathsf{V}$ and $\mathsf{R}$ are arguing whether $m'\in X_0$.
    $\mathsf{V}$ wins (automatically) iff $m'\in X_0$.

    Before we consider the descriptive complexity of the payoffs of these games, consider the following example:
    \begin{align*}
        \mathsf{I}  \;\;& \tuple{m,3} \;\; \phantom{\text{yes} \;\; X^f_2}\;\; \text{challenge } m'\;\; X^f_1 \;\; \phantom{\text{challenge } m''} \;\; X^f_0 \\
        \mathsf{II} \;\;& \phantom{\tuple{m,i}} \;\; \text{yes} \;\; X^f_2\;\; \phantom{\text{challenge } m'\;\; X^f_1} \;\; \text{challenge } m'' \phantom{\;\; X^f_0}
    \end{align*}
    Note that in this game $\mathsf{I}$ and $\mathsf{II}$ play both roles $\mathsf{V}$ and $\mathsf{R}$.

    Since we want a game with payoff in Cantor space, the players cannot directly play natural numbers.
    Code a play of a natural number $m$ by a sequence of plays consisting of $m$ many zeroes and ending with a one.
    We identify finite sequences of natural numbers with their codes.
    As the players must alternate playing, when it is a player's turn to play a natural number (\emph{i.e}, a sequence of zeroes ending with a one), the other player is on stand-by, and must plays only zeroes.
    In particular, the example play above becomes:
    \begin{align*}
        \mathsf{I}\;\;&0^{\tuple{m,3}}1 \phantom{0^{1}1 0^{X^f_{2}}1} 0^{m'}1 0^{X^f_{1}}1 \phantom{0^{m''}1} 0^{X^f_{0}}  \\
        \mathsf{II}\;\;&\phantom{0^{\tuple{m,3}}1} 0^{1}1 0^{X^f_{2}}1 \phantom{0^{m'}1 0^{X^f_{1}}1} 0^{m''}1 \phantom{0^{X^f_{0}}} \\
    \end{align*}
    We can describe the restrictions on each move using a $\Sigma^0_1$ set or a $\Pi^0_1$ set, depending on which player must follow the rule.
    Also, we can check if the witnesses given by $\mathsf{V}$ are right with a $\Pi^0_1$ condition.
    Therefore we can write the game above as a finite difference of $2^{n+2}$ open sets on Cantor space, and it is determined. (The complexity bound of the game need not be strict, as we assume $\forall n.(\Sigma^0_1)_n\text{-}\mathsf{Det}^*$.)

    We claim that $\mathsf{I}$ cannot have a winning strategy inside $\mathcal{M}$.
    Suppose $\sigma\in\mathcal{M}$ is a winning strategy for $\mathsf{I}$.
    Now, after $\mathsf{I}$ plays $\sigma(\tuple{}) = \tuple{m,i}$, $\mathsf{II}$ must try to show that $m\in X_i$ or ask $\mathsf{I}$ to do so; for either choice the subsequent game is symmetric, only changing who is going to play $\mathsf{V}$ and $\mathsf{R}$.
    Therefore if $\mathsf{I}$ wins as $\mathsf{V}$, then $\mathsf{II}$ can also win as $\mathsf{V}$, using essentially the same strategy.
    Therefore $\mathsf{II}$ can use a (slightly modified variant of) $\sigma$ to defeat $\sigma$.
    And so $\sigma$ cannot be winning.

    Let $\tau$ be a winning strategy for $\mathsf{II}$ inside $\mathcal{M}$ and define $Z=\{\tuple{m, i} \mid \tau(\tuple{m,i}) = 1\}$.
    We can show that $(Z)_0 = X$ and $(Z)_{i+1} = \mathrm{TJ}((Z)_i)$ for $i<n$ by induction inside $\mathcal{M}$.
    As $\mathcal{M}$ is a strict $\beta$-model and $(Z)_{i+1} = \mathrm{TJ}((Z)_i)$ is a $\Pi^0_1$ statement, we also have that $(Z)_0 = X$ and $(Z)_{i+1} = \mathrm{TJ}((Z)_i)$ hold outside of $\mathcal{M}$.
    We conclude that $\mathsf{ACA}_0'$ holds.
\end{proof}

\begin{proposition} \label{prop::cantor_reduce_complexity}
    Over $\mathsf{ACA}_0$, $\mathsf{ACA}_0'$ implies $\forall n.(\Sigma^0_1)_n\text{-}\mathsf{Det}^*$.
    Therefore, $\forall n.(\Sigma^0_1)_n\text{-}\mathsf{Det}^*$ is equivalent to a $\Pi^1_2$ sentence.
\end{proposition}
\begin{proof}
    We proved that $\forall n.(\Sigma^0_1)_n\text{-}\mathsf{Det}^*$ implies $\mathsf{ACA}_0'$ in Proposition \ref{prop::sigma01det-n_implies_ACA0prime}.
    We show that the converse holds by a generalization of Theorem 3.7 of \cite{nemoto2007infinite}.
    As $\mathsf{ACA}_0'$ is a $\Pi^1_2$-sentence, we will have the proposition.

    Fix a difference $\varphi(f)$ of $2k$ many $\Sigma^0_1$ formulas with set parameter $X$. $\varphi(F)$ can be thought as a formula
    \[
        \exists n.R_0(f[n]) \land \psi_0(f) \land \cdots \exists n.R_{k-1}(f[n]) \land \psi_{k-1}(f),
    \]
    where the $R_l$ are $\Pi^0_0$-formulas and the $\psi_l$ are $\Pi^0_1$-formulas, for all $l\leq k$.
    Define the sets:
    \begin{itemize}
        \item $W_0 := \{u\in 2^{<\mathbb{N}} \mid \exists v\subseteq u.R_0(v)\text{ and $\mathsf{I}$ wins $\psi_0(f)$ starting at $u$}\}$; and
        \item $W_{l+1} := \{u\in W_l \mid \exists v\subseteq u.R_{l+1}(v)\text{ and $\mathsf{I}$ wins $\psi_{l+1}(f)$ starting at $u$ and playing inside $W_l$}\}$.
    \end{itemize}
    Define the game $\psi^*(f) \leftrightarrow \exists n.W_{k-1}(f[n])$.
    By Theorem 3.6 of \cite{nemoto2007infinite}, $W_0$ is $\Pi^0_1$ and each $W_{l+1}$ is $\Pi^0_1$ in $W_l$ as a parameter.
    By bounded induction on $n$, we can compute $W_{k-1}$ from the $k$th jump of $X$. Analogously to Claim 3.7.1, from a strategy for $\mathsf{I}$ in $\psi^*$ we can compute a strategy for $\mathsf{I}$ in the game with payoff $\varphi(f)$, and the same holds with $\mathsf{II}$ in place of $\mathsf{I}$.
    As $\psi^*$ is a $\Sigma^0_1$ game on Cantor space, it is determined.
    Therefore the game with payoff $\varphi$ is also determined.
    As this argument holds for any $k\in\mathbb{N}$, $\forall n.(\Sigma^0_1)_n\text{-}\mathsf{Det}^*$ is true.
\end{proof}

We do not explicitly consider the differences of $\Sigma^0_2$ sets on Cantor space.
Nemoto \emph{et al.} \cite{nemoto2007infinite} proved that, for $1<k\leq \omega$,  $(\Sigma^0_2)_k$-$\mathsf{Det}$ and $(\Sigma^0_2)_{k-1}$-$\mathsf{Det}^*$ are equivalent over $\mathsf{RCA}_0$.
We can adapt their proof to show:
\begin{proposition}
    $\forall n.(\Sigma^0_2)_n$-$\mathsf{Det}$ and $\forall n.(\Sigma^0_2)_n$-$\mathsf{Det}^*$ are equivalent over $\mathsf{RCA}_0$.
\end{proposition}
\begin{proof}
    Every game on Cantor space can be seen as a game on Baire space by adding a $\Pi^0_1$ condition: $\mathsf{I}$ and $\mathsf{II}$ play only zeroes and ones.
    Therefore a game in Cantor space which payoff is in $(\Sigma^0_2)_n$ is still a $(\Sigma^0_2)_n$ game in Baire space.
    So $\forall n.(\Sigma^0_2)_n$-$\mathsf{Det}$ implies $\forall n.(\Sigma^0_2)_n$-$\mathsf{Det}^*$.

    By Lemma 4.2 of \cite{nemoto2007infinite}, if a game on Baire space has payoff $A\in (\Sigma^0_2)_n$, then there is a game on Cantor space with payoff $A^*\in (\Sigma^0_2)_{n+2}$ such that $\mathsf{I}$($\mathsf{II}$) has a winning strategy in $A$ iff $\mathsf{I}$($\mathsf{II}$) has a winning strategy in $A^*$.
    Therefore $\forall n.(\Sigma^0_2)_n$-$\mathsf{Det}^*$ implies $\forall n.(\Sigma^0_2)_n$-$\mathsf{Det}$.
\end{proof}

For determinacy on Baire space the following results are known: Steel \cite{steel1977phdthesis} proved that $\Delta^0_1\text{-}\mathsf{Det}$ and $\Sigma^0_1\text{-}\mathsf{Det}$ are equivalent to $\mathsf{ATR}_0$ over $\mathsf{ACA}_0$; and, Tanaka \cite{tanaka1990weak} proved that $(\Sigma^0_1)_2\text{-}\mathsf{Det}$ is equivalent to $\Pi^1_1\text{-}\mathsf{CA}_0$ over $\mathsf{ATR}_0$.
Similar to finite differences of open sets on Cantor space, where $\forall n.(\Sigma^0_1)_n\text{-}\mathsf{Det}^*$ proves $\mathsf{ACA}_0'$, we can use the determinacy of arbitrary finite differences of open sets in Baire space to prove $\Pi^1_1\text{-}\mathsf{CA}_0'$.
\begin{proposition} \label{prop::sigma01det-n_implies_seq_of_beta-models}
    Over $\mathsf{ACA}_0$, $\forall n.(\Sigma^0_1)_n\text{-}\mathsf{Det}$ implies $\Pi^1_1\text{-}\mathsf{CA}_0'$.
\end{proposition}
\begin{proof}
    We first prove that $\Pi^1_1\text{-}\mathsf{CA}_0'$ is equivalent to  ``for all $X\subseteq \mathbb{N}$ and $n\in\mathbb{N}$, the $n^\mathrm{th}$ iterated hyperjump $\mathrm{HJ}^n(X)$ of $X$ exists''.
    The proof of this equivalence is essentially the proof of Theorem VII.2.9 of \cite{simpson2009sosoa}.
    Fix a sequence of coded $\beta$-models $X_0 \in X_1 \in \cdots \in X_n$ with $X\in X_0$. We can define $\mathrm{HJ}(X)$ arithmetically using $X_0$ as a parameter, so $\mathrm{HJ}(X)$ exists.
    Since $X_0\in X_1\models \mathsf{ACA}_0$.
    Again, we can define $\mathrm{HJ}^2(X)$ arithmetically using $X_0$ as a parameter.
    Repeating this process boundedly many times, we can define $\mathrm{HJ}^{n+1}(X)$.
    On the other hand, suppose the hyperjumps $\mathrm{HJ}(X), \dots, \mathrm{HJ}^n(X)$ exist.
    From $\mathrm{HJ}(X)$ we can define a $\beta$-model $X_0\ni X$. By bounded induction, given $X_i$ and $\mathrm{HJ}^{i+1}(X)$, we can define $X_{i+1}$ with $X_i\in X_{i+1}$ and $\mathrm{HJ}^i(X)\in X_{i+1}$.
    Therefore there is a sequence $X_0 \in \cdots \in X_n$ with $X\in X_0$.

    Suppose $\forall n.(\Sigma^0_1)_n\text{-}\mathsf{Det}$ holds and fix $n\in\mathbb{N}$ and $X\subseteq \mathbb{N}$. We prove the existence of the sequence $\tuple{\mathrm{HJ}(X), \dots, \mathrm{HJ}^n(X)}$.

    Similar to Proposition \ref{prop::sigma01det-n_implies_ACA0prime}, we claim that there is a set $\tilde{\tau}$ computing winning strategies for all $(\Sigma^0_1)_{k}$ games on Baire space with only $X$ as a parameter, for all $k\in\mathbb{N}$.
    Consider the following game: $\mathsf{I}$ chooses an index $e$ for the payoff of a $(\Sigma^0_1)_{k}$ game; $\mathsf{II}$ answers with their choice of role in the game with index $e$; then $\mathsf{I}$ and $\mathsf{II}$ play the game with index $e$; whoever wins the subgame wins the whole game.
    The payoff of this game has complexity $(\Sigma^0_1)_{2k}$.
    As we have $\forall n.(\Sigma^0_1)_n\text{-}\mathsf{Det}^*$, the game above is determined.
    Since $\mathsf{I}$ cannot win, $\mathsf{II}$ has a winning strategy $\tilde{\tau}$; this $\tilde{\tau}$ computes winning strategies for all $(\Sigma^0_1)_{k}$ games.

    Since we have $\forall n.(\Sigma^0_1)_n\text{-}\mathsf{Det}$, we also have $(\Sigma^0_1)_2\text{-}\mathsf{Det}$, and so $\Pi^1_1\text{-}\mathsf{CA}_0$ holds.
    Let $\mathcal{M}$ be a $\beta$-model including $X$ and a universal winning strategy $\tilde{\tau}$ for $(\Sigma^0_1)_{2^{n+2}}$ games with $X$ as the only parameter.
    The $\beta$-model $\mathcal{M}$ satisfies $\mathsf{ATR}_0$ and induction.

    We consider games similar to the ones in the proof of Proposition \ref{prop::sigma01det-n_implies_ACA0prime}.
    Player $\mathsf{I}$ starts by playing $\tuple{m,i}$ with $i\leq n$, asking whether $m\in \mathrm{HJ}^i(X)$.
    Then $\mathsf{II}$ plays 1 to answer `yes' and 0 to answer `no'.
    If $\mathrm{II}$ play 1, then they play the role of $\mathsf{V}$ (Verifier), otherwise they play the role of $\mathsf{R}$ (Refuter).

    If $i>0$, $\mathsf{V}$ must now play the characteristic function $\chi_{\mathrm{HJ}^{i-1}(X)}(m')$ of $\mathrm{HJ}^{i-1}(X)$ and a function $f$ witnessing that $m\in \mathrm{HJ}^{i}(X)$.
    While $\mathsf{V}$ is playing, $\mathsf{R}$ may contest $\mathsf{V}$'s choice of some $\chi_{\mathrm{HJ}^{i-1}(X)}(m')$.
    If $\mathsf{V}$ stated that $m'\not\in \mathrm{HJ}^{i-1}(X)$, then the players exchange roles; otherwise the roles stay the same.
    The players then proceed to discuss whether $m'\in \mathrm{HJ}^{i-1}(X)$.
    In case $\mathsf{R}$ never contests, $\mathsf{V}$ wins iff $\pi^0_1(m,f,\mathrm{HJ}^{i-1}(X))$ holds.
    If $i=0$, $\mathsf{V}$ wins iff $m\in X$.
    This game can be described by a boolean combination of $2^{n+2}$-many $\Sigma^0_1$-formulas, and so it is determined. (The complexity bound of the game need not be strict, as we assume $\forall n.(\Sigma^0_1)_n\text{-}\mathsf{Det}$.)

    We claim that $\mathsf{I}$ cannot have a winning strategy inside $\mathcal{M}$.
    As in Proposition \ref{prop::sigma01det-n_implies_ACA0prime}, $\mathsf{I}$ cannot have a winning strategy.
    Let $\tau$ be a winning strategy for $\mathsf{II}$.
    Define $Z = \{ \tuple{m,i} \mid \tau(\tuple{m,i}) = 1\}$.
    We can show that $(Z)_0=X$ and $(Z)_{i+1} = \mathrm{HJ}(\mathrm{HJ}^i(X))$ for $i<n$ by induction inside $\mathcal{M}$.
    As $\mathcal{M}$ is a $\beta$-model, the same holds outside of $\mathcal{M}$.
    We conclude that $\Pi^1_1\text{-}\mathsf{CA}_0'$ holds.
\end{proof}

MedSalem and Tanaka \cite{medsalem2008weak} proved that $(\Sigma^0_2)_k\text{-}\mathsf{Det}$ and $[\Sigma^1_1]^k\text{-}\mathsf{ID}$ are equivalent over $\mathsf{ATR}_0$, for $0<k<\omega$.
\begin{proposition}
    \label{prop::ind-k_and_determinacy-forall-ver}
    Over $\mathsf{ATR}_0$, $\forall n.(\Sigma^0_2)_n\text{-}\mathsf{Det}$ implies $\forall n.[\Sigma^1_1]^n\text{-}\mathsf{ID}$.
\end{proposition}
\begin{proof}
    Fix $k\in\mathbb{N}$ and $X\subseteq \mathbb{N}$.
    Suppose $\forall n.(\Sigma^0_2)_n\text{-}\mathsf{Det}$ is true.
    Let $\tuple{\Gamma_0,\dots, \Gamma_{n-1}}$ be a sequence of (indices of) $\Sigma^1_1$-inductive operators.
    We show that the set $V_n$ inductively defined by $\tuple{\Gamma_0,\dots, \Gamma_{n-1}}$ exists.

    As in Propositions \ref{prop::sigma01det-n_implies_ACA0prime} and \ref{prop::sigma01det-n_implies_seq_of_beta-models}, we claim that there is a set $\tilde{\tau}$ computing universal winning strategies for all $(\Sigma^0_2)_{k}$ games with only $X$ as a parameter, for all $k\in\mathbb{N}$.
    Consider the following game: $\mathsf{I}$ chooses an index $e$ for the payoff of a $(\Sigma^0_2)_{k}$ game; $\mathsf{II}$ answers with their choice of role in the game with index $e$; then $\mathsf{I}$ and $\mathsf{II}$ play the game with index $e$; whoever wins the subgame wins the whole game.
    The payoff of this game has complexity $(\Sigma^0_2)_{2k}$.
    As we have $\forall n.(\Sigma^0_2)_n\text{-}\mathsf{Det}$, the game above is determined.
    Since $\mathsf{I}$ cannot win, $\mathsf{II}$ has a winning strategy $\tilde{\tau}$; this $\tilde{\tau}$ computes winning strategies for all $(\Sigma^0_2)_{k}$ games.

    Since we have $\forall n.(\Sigma^0_2)_n\text{-}\mathsf{Det}$, we also have $(\Sigma^0_1)_2\text{-}\mathsf{Det}$, and so $\Pi^1_1\text{-}\mathsf{CA}_0$ is true.
    Let $\mathcal{M}$ be a $\beta$-model including $X$ and a universal winning strategy $\tilde{\tau}$ for $(\Sigma^0_2)_{{n}^3}$ games.
    The $\beta$-model $\mathcal{M}$ satisfies $\mathsf{ATR}_0$ and induction.

    We use an unfolded version of a proof by MedSalem and Tanaka \cite{medsalem2008weak}*{Theorem 3.3} to show that the set $V_{n}$ inductively defined by $\tuple{\Gamma_0,\dots, \Gamma_{n-1}}$ exists inside $\mathcal{M}$.
    They prove that, for all $n\in\omega$, $[\Sigma^1_1]^n\text{-}\mathsf{ID}$ follows from $(\Sigma^0_2)_n\text{-}\mathsf{Det}$ using $\mathsf{ATR}_0$ and induction on $n$.
    We sketch how to unfold their proof to show that $[\Sigma^{1,X}_1]^{n}\text{-}\mathsf{ID}$ follows from $(\Sigma^{0,X}_2)_{{n}^3}\text{-}\mathsf{Det}$.
    Here, $\Sigma^{i,X}_j$ denotes the set of $\Sigma^i_j$ formulas whose only set parameter is $X$.
    Let $V_i$ be the set inductively defined by $\tuple{\Gamma_0, \dots, \Gamma_{i-1}}$, for $i= 1, \dots n$.
    To show the existence of the set $V_{n}$ , we use the set $V_{{n}-1}$ and a $(\Sigma^0_2)_{n}$ game.
    Unfolding the definitions of $V_{{n}-1}$, we can show the existence of $V_{n}$ using $V_{{n}-2}$ and a $(\Sigma^0_2)_{{n}-1}\land (\Sigma^0_2)_{n}$ game.
    Repeatedly unfolding the $V_i$, we an prove the existence of $V_{n}$ using a $\Sigma^0_2 \land (\Sigma^0_2)_2 \land \cdots \land (\Sigma^0_2)_{n}$ game.
    Furthermore, we can show that the payoff of this game is $(\Sigma^0_2)_{{n}^3}$. (The complexity bound of the game need not be strict, as we assume $\forall n.(\Sigma^0_2)_n\text{-}\mathsf{Det}$.)

    Furthermore, the statement ``$V_{n}$ is the set inductively defined by $\tuple{\Gamma_0, \dots, \Gamma_{n-1}}$'' is a boolean combination of $\Pi^1_1$-sentences with $X, V_{n}$ as the only parameters.
    As $\mathcal{M}$ is a $\beta$-models, if $V_{n}$ is the set inductively defined by $\tuple{\Gamma_0, \dots, \Gamma_{n-1}}$ inside $\mathcal{M}$, then $V_{n}$ is also the set inductively defined by $\tuple{\Gamma_0,\dots, \Gamma_{n-1}}$ outside $\mathcal{M}$, as we wanted to show.
    Since the argument above holds for any $k\in\mathbb{N}$, $X\subseteq\mathbb{N}$ and $\tuple{\Gamma_0,\dots, \Gamma_{n-1}}$, we have that $\forall n.[\Sigma^1_1]^n\text{-}\mathsf{ID}$ holds.
\end{proof}
The reversals for Propositions \ref{prop::sigma01det-n_implies_seq_of_beta-models} and \ref{prop::ind-k_and_determinacy-forall-ver} will be proved below.

\section{Sequences of $\beta_k$ models and reflection}\label{section::sequences_of_beta_models}
We define now a formula $\psi_e(i,n)$ which states that there are sequences with length $n$ of increasing coded models which are $\beta_i$-submodels of each other and where the last model is a $\beta_e$-submodel of the ground model $\groundmodel$:
\begin{align*}
     X \in \;&Y_0 \; \in  \;\;\cdots \; \in \;\;\; Y_n, \\
           &Y_0 \subseteq_{\beta_i}\cdots \subseteq_{\beta_i} Y_n \subseteq_{\beta_e} \groundmodel.
\end{align*}
For each $e\in \omega$, we define $\psi_e(i,n)$ by
\[
    \forall X \exists Y_0, \dots, Y_n \forall k\leq n.\left \{
    \begin{array}{l}
        X\in Y_0 \;\land\\
        Y_k\in Y_{k+1} \;\land\\
        Y_k \models \mathsf{ACA}_0 \;\land\\
        Y_k \subseteq_{\beta_i} Y_{k+1} \;\land\\
        Y_n \subseteq_{\beta_e} \groundmodel.
    \end{array}\right.
\]
Note that each $\psi_e(i,n)$ is a $\Pi^1_{e+2}$-formula.
The following is an immediate consequence of our definition: $\psi_e(i,n)$ is downwards closed, {\em i.e.}, if $\psi_e(i,n)$ holds and $n'\leq n, i'\leq i, e'\leq e$ then $\psi_{e'}(i',n')$ also holds.

In this section, we prove:
\begin{theorem}
    \label{thm::psi_and_ref}
    Over $\mathsf{ACA}_0$, if $e\leq i$ then $\forall n.\psi_e(i,n)$ is equivalent to $\Pi^1_{e+2}$-$\refl(\sigmaoneisdc)$.
\end{theorem}
\noindent We justify our choice of variables by noting that $i$ stands for `internal' and $e$ for `external', on both sides of the theorem above.

\begin{lemma}\label{lem::refl_sigmaoneisdc_proves_psi}
    Over $\mathsf{ACA}_0$, $\Pi^1_{e+2}$-$\refl(\sigmaoneisdc)$ proves $\forall n.\psi_e(i,n)$ when $e\leq i$.
\end{lemma}
\begin{proof}
    In this proof we denote $\sigmaoneisdc$ by $T_i$.
    We have $\bew{T_i}{\psi_i(i,0)}$ and $\bew{T_i}{\forall n.\psi_i(i,n) \rightarrow \psi_i(i,n+1)}$.
    So $\bew{T_i}{\psi_i(i,n)}$ holds for each $n\in \mathbb{N}$.
    By reflection, $\psi_i(i,n)$ holds for any $n\in\mathbb{N}$.
    Thus $\forall n.\psi_i(i,n)$ holds.
    As $\psi$ is downwards closed, $\forall n.\psi_e(i,n)$ holds when $e\leq i$.
\end{proof}

\begin{lemma}\label{lem::psi_proves_refl_sigmaoneisdc}
    Over $\mathsf{ACA}_0$, $\forall n.\psi_e(i,n)$ proves $\Pi^1_{e+2}$-$\refl(\sigmaoneisdc)$ when $e\leq i$.
\end{lemma}
\begin{proof}
    Assume that $\Pi^1_{e+2}$-$\refl(\sigmaoneisdc)$ is false.
    That is, there is an $\mathcal{L}_2$-formula $\theta(X)\in \Sigma^1_{e+1}$ such that $\bew{\sigmaoneisdc}{\forall X.\theta(X)}$ holds but $\forall X.\theta(X)$ is false.
    Therefore, there is $X_0$ such that $\neg\theta(X_0)$ holds.

    Let $\mathcal{L}_1'$ be the language $\mathcal{L}_1$ of first-order arithmetic plus a unary predicate symbol $A$.
    Denote the domain of an $\mathcal{L}_1'$ structure by $M$.
    In the definition below, $\mathcal{M}_j$ is $(M,\{(A_j)_i i\in M\})$, for $j\in \mathbb{N}$.
    Define the $\mathcal{L}_1'$-theory $T$ by the following axioms:
    \begin{enumerate}
        \item $M$ is a discrete ordered semiring;

        \item $\mathcal{M}_j \models \mathsf{ACA}_0$ for all $j\in \mathbb{N}$;

        \item $(\mathcal{M}_{j+1})_0 = \mathcal{M}_j$, formally:
        \[
            \forall m .m\in M_j \leftrightarrow (M_{j+1})_0;
        \]


        \item $\mathcal{M}_j \subseteq_{\beta_i} \mathcal{M}_{j+1}$ for all $j\in \mathbb{N}$:
        \[
            \forall e_0 \forall s.(\exists m.m\in (M_j)_s) \to (\mathcal{M}_j\models \pi^1_i(e_0,(M_j)_s) \leftrightarrow \mathcal{M}_{j+1}\models \pi^1_i(e_0,(M_j)_s)),
        \]
        where $\pi^1_i$ is a universal lightface $\Pi^1_i$-formula;

        \item $\mathcal{M}_j$ satisfies $\neg\theta(X_0)$; and

        \item $X_0 = (\mathcal{M}_0)_0$, that is, for all $n\in\mathbb{N}$, $n\in X_0$ if and only if $x\in (\mathcal{M}_0)_0$.
    \end{enumerate}

    Fix a finite subtheory $T'$ of $T$.
    Let $j_0$ be the greatest index $j$ of a coded model $\mathcal{M}_j$ occurring in formulas of $T'$.
    $\psi_e(i,j_0)$ implies that there is a sequence
    \[
        X_0\in Y_0 \subseteq_{\beta_i} \cdots \subseteq_{\beta_i} Y_{j_0} \subseteq_{\beta_e} \groundmodel
    \]
    of $j_0$ many coded models.
    Setting $A_j = Y_j$ for $j\leq j_0$ and $A_j = \emptyset$ for $j>j_0$, we have a model witnessing the consistency of $T'$.
    By compactness, $T$ is also consistent, so there is a model $\mathcal{M}=(M,A)$ of $T$.

    Now consider the model $\mathcal{M}_\omega = (M, \{(A_j)_i \mid i\in M, j\in\mathbb{N}\})$.
    $\mathcal{M}_\omega$ satisfies $\sigmaoneisdc$ as it is closed under taking $\beta_i$ models: if $Y\in \mathcal{M}_\omega$ then $Y$ is in some $\beta_i$-model $\mathcal{M}_j$ which is also in $\mathcal{M}_\omega$.
    So $\mathcal{M}_\omega$ is a model of $\theta(X_0)$.
    On the other hand, each $\mathcal{M}_j$ is a $\beta_i$-submodel of $\mathcal{M}_\omega$ since $\mathcal{M}_j \subseteq_{\beta_i} \mathcal{M}_{j+1}$ for all $j\in \mathbb{N}$.
    As $\neg\theta(X_0)$ is $\Pi^1_{e+1}$, it can be written as $\forall Y \theta'(X_0,Y)$ with $\theta'\in\Sigma^1_e$.
    Let $Y_0\in \mathcal{M}_\omega$, then there is $j\in\mathbb{N}$ such that $Y_0\in \mathcal{M}_j$.
    Since $\neg \theta(X_0)$ holds in $\mathcal{M}_j$, $\theta'(X_0,Y_0)$ also holds in $\mathcal{M}_j$.
    As $e\leq i$, $\theta'\in\Sigma^1_e$ and $\mathcal{M}_j \subset_{\beta_i} \mathcal{M}_\omega$, we have that $\theta'(X_0,Y_0)$ also holds in $\mathcal{M}_\omega$.
    Since the argument above holds for arbitrary $Y_0\in \mathcal{M}_\omega$, we have that $\forall Y \theta'(X_0,Y)$ holds in $\mathcal{M}_\omega$.
    That is, $\neg\theta(X_0)$ holds in $\mathcal{M}_\omega$.
    Therefore both $\forall X.\theta(X)$ and $\exists X.\neg\theta(X)$ hold in $\mathcal{M}$, a contradiction.
\end{proof}

Theorem \ref{thm::psi_and_ref} follows from Lemmas \ref{lem::refl_sigmaoneisdc_proves_psi} and \ref{lem::psi_proves_refl_sigmaoneisdc}.

\begin{corollary}
    \label{cor::comp_and_ref}
    Over $\mathsf{ACA}_0$,
    \begin{enumerate}
        \item $\forall n.\psi_1(1,n)$ is equivalent to $\Pi^1_3$-$\refl(\Pi^1_1$-$\mathsf{CA}_0)$; and
        \item $\forall n.\psi_1(2,n)$ is equivalent to $\Pi^1_3$-$\refl(\Pi^1_2$-$\mathsf{CA}_0)$.
    \end{enumerate}
\end{corollary}
\begin{proof}
    By Theorem VII.6.9 of \cite{simpson2009sosoa}, $\mathrm{Strong} \; \Sigma^1_1 \text{-}\mathsf{DC}_0$ is equivalent $\Pi^1_1$-$\mathsf{CA}_0$, and $\mathrm{Strong} \; \Sigma^1_2 \text{-}\mathsf{DC}_0$ is equivalent $\Pi^1_2$-$\mathsf{CA}_0$ (over $\mathsf{ACA}_0$).
\end{proof}

We can now prove:
\begin{theorem}
    Over $\mathsf{ACA}_0$, $\forall n.(\Sigma^0_1)_n$-$\mathsf{Det}$, $\Pi^1_3$-$\refl(\Pi^1_1$-$\mathsf{CA}_0)$, and $\Pi^1_1\text{-}\mathsf{CA}_0'$ are pairwise equivalent.
\end{theorem}
\begin{proof}
    By Proposition \ref{prop::sigma01det-n_implies_seq_of_beta-models},  $\forall n.(\Sigma^0_1)_n\text{-}\mathsf{Det}$ implies $\Pi^1_1\text{-}\mathsf{CA}_0'$.
    Also, $\Pi^1_1\text{-}\mathsf{CA}_0'$ and $\forall n.\psi_1(1,n)$ are equivalent.
    By Corollary \ref{cor::comp_and_ref}.1, $\forall n.\psi_1(1,n)$ is equivalent to $\Pi^1_3$-$\refl(\Pi^1_1$-$\mathsf{CA}_0)$.
    Finally, since $\Pi^1_1$-$\mathsf{CA}_0$ proves $(\Sigma^0_1)_2$-$\mathsf{Det}$ and $\forall n. (\Sigma^0_1)_n$-$\mathsf{Det} \to (\Sigma^0_1)_{n+1}$-$\mathsf{Det}$, we can prove $\forall n.(\Sigma^0_1)_n\text{-}\mathsf{Det}$ from $\Pi^1_3$-$\refl(\Pi^1_1$-$\mathsf{CA}_0)$.
\end{proof}

We modify the $\psi_e$ to get a similar result for $\mathsf{ACA}_0'$.
Let $\psi'(n)$ be defined by
\[
    \forall X \exists Y_0, \dots, Y_n \forall k\leq n.\left \{
    \begin{array}{l}
        X \in Y_0 \; \land\\
        (Y_k)' \in Y_{k+1} \; \land\\
        Y_k \models \mathsf{RCA}_0.\\
    \end{array}\right.
\]
One can show that:
\begin{lemma}
    Over $\mathsf{ACA}_0$, $\mathsf{ACA}_0'$ is equivalent to $\forall n.\psi'(n)$.
\end{lemma}
\begin{proof}[Proof sketch.]
    Fix $X$. First suppose $\mathsf{ACA}_0'$. Given $n\in\mathbb{N}$, we can compute the first $n$ jumps of $X$. For $k\leq n$, let $Y_k$ be the collection of sets computable from $\mathrm{TJ}^n(X)$. Then $Y_0, \dots, Y_n$ satisfy $\psi(n)$. Now, suppose $\forall n.\psi'(n)$ holds. We can extract $\mathrm{TJ}^n(X)$ from $Y_n$ if $Y_0, \dots, Y_n$ witness $\psi'(n)$.
\end{proof}

Similar to Theorem \ref{thm::psi_and_ref}, we have:
\begin{theorem}
    Over $\mathsf{ACA}_0$, $\forall n.\psi'(n)$ is equivalent to $\Pi^1_2$-$\refl(\mathsf{ACA}_0)$.
\end{theorem}
\begin{proof}
    As $\mathsf{ACA}_0$ proves that the Turing jump of any set exists, $\Pi^1_2$-$\refl(\mathsf{ACA}_0)$ proves $\mathsf{ACA}_0'$ and $\forall n.\psi'(n)$.

    Now, let $\theta(X)\in \Sigma^1_1$ be an $\mathcal{L}_2$ formula such that $\bew{\mathsf{ACA}_0}{\forall X.\theta(X)}$ and there is $X_0$ such that $\neg\theta(X_0)$ holds.
    Let the language $\mathcal{L}_1'$ be as in the proof of Lemma \ref{lem::psi_proves_refl_sigmaoneisdc} and define an $\mathcal{L}_1'$-theory $T$ by:
    \begin{enumerate}
        \item $M$ is a discrete ordered semiring;
        \item $\mathcal{M}_j\models\mathsf{RCA}_0$ for all $j\in \mathbb{N}$;
        \item $\mathcal{M}_j\subseteq \mathcal{M}_{j+1}$ for all $j\in \mathbb{N}$;
        \item $(\mathcal{M}_j)'\in \mathcal{M}_{j+1}$ for all $j\in \mathbb{N}$;
        \item $\mathcal{M}_j\models \neg \theta(X_0)$ for all $j\in \mathbb{N}$; and
        \item $X_0 = (\mathcal{M}_0)_0$, that is, for all $n\in\mathbb{N}$, $n\in X_0$ if and only if $x\in (\mathcal{M}_0)_0$.
    \end{enumerate}
    Again, $\mathcal{M}_j$ is $(M,\{(A_j)_i\mid i\in M\})$.

    Now, $\forall n.\psi'(n)$ supplies a model for any finite subtheory of $T$.
    By compactness, there is a model $\mathcal{M} = (M,A)$ of $T$.
    Define the coded model $\mathcal{M}_\omega$ by $(M,\{(A_j)_i \mid i\in M, j\in \mathbb{N}\})$.
    Since $\mathcal{M}_\omega$ is closed under Turing jumps, it is a model of $\mathsf{ACA}_0$, and thus $\mathcal{M}_\omega\models \theta(X_0)$.
    But by the definition of $T$, $\mathcal{M}_\omega\models \neg\theta(X_0)$.
    This is a contradiction.
    Therefore if $\bew{\mathsf{ACA}_0}{\forall X.\theta(X)}$ holds, so must $\forall X.\theta(X)$.
\end{proof}

We can then use Propositions \ref{prop::sigma01det-n_implies_ACA0prime} and \ref{prop::cantor_reduce_complexity} to show:
\begin{theorem}
    Over $\mathsf{ACA}_0$, $\forall n.(\Sigma^0_1)_n$-$\mathsf{Det}^*$, $\Pi^1_2$-$\refl(\mathsf{ACA}_0)$ and $\mathsf{ACA}_0'$ are pairwise equivalent.
\end{theorem}

During the RIMS 2021 Proof Theory Workshop, Toshiyasu Arai asked about the characterization of the existence of sequences of coded models of transfinite length.

\section{The $\Pi^1_2$-$\mathsf{CA}_0$ Case}
\label{section::the_pi12-ca_case}
In this section, we prove:
\begin{theorem}
    \label{thm::main-pi12-theorem}
    Over $\mathsf{ACA}_0$, $\forall n.(\Sigma^0_2)_n$-$\mathsf{Det}$, $\Pi^1_3$-$\refl(\Pi^1_2$-$\mathsf{CA}_0)$, and $\forall n.[\Sigma^1_1]^n\text{-}\mathsf{ID}$ are pairwise equivalent.
\end{theorem}
Like in the $\forall n.(\Sigma^0_1)_n$-$\mathsf{Det}$ and $\forall n.(\Sigma^0_1)_n$-$\mathsf{Det}^*$ cases, one half of this theorem has a straight proof:
\begin{lemma}
    Over $\mathsf{ACA}_0$, $\Pi^1_3$-$\refl(\Pi^1_2$-$\mathsf{CA}_0)$ proves $\forall n.(\Sigma^0_2)_n$-$\mathsf{Det}$.
\end{lemma}
\begin{proof}
    $\Pi^1_2\text{-}\mathsf{CA}_0$ proves $\Sigma^0_2$-$\mathsf{Det}$ and, for all $n\in\omega$, $(\Sigma^0_2)_n\text{-}\mathsf{Det} \to (\Sigma^0_2)_{n+1}\text{-}\mathsf{Det}$.
    Formalizing these proofs inside $\mathsf{ACA}_0$, we have that
    \[
        \bew{\Pi^1_2\text{-}\mathsf{CA}_0}{\Sigma^0_2\text{-}\mathsf{Det}}
    \]
    and
    \[
        \bew{\Pi^1_2\text{-}\mathsf{CA}_0}{(\Sigma^0_2)_{n}\text{-}\mathsf{Det}\to (\Sigma^0_2)_{n}\text{-}\mathsf{Det}}.
    \]
    By $\Sigma^0_1$-induction, we have
    \[
        \forall n.\bew{\Pi^1_2\text{-}\mathsf{CA}_0}{(\Sigma^0_2)_n\text{-}\mathsf{Det}}.
    \]
    In particular, for any $n\in\mathbb{N}$, $\bew{\Pi^1_2\text{-}\mathsf{CA}_0}{(\Sigma^0_2)_n\text{-}\mathsf{Det}}$.
    So $\Pi^1_3$-$\refl(\Pi^1_2$-$\mathsf{CA}_0)$ implies $(\Sigma^0_2)_n\text{-}\mathsf{Det}$.
\end{proof}

We will use multiple $\Sigma^1_1$-inductive definitions to show the other half of Theorem \ref{thm::main-pi12-theorem}.
Before doing so, we review the idea behind $[\Sigma^1_1]^k$-$\mathsf{ID}$.
Let $\Gamma_0, \Gamma_1$ be $\Sigma^1_1$-inductive operators.
Note that we do not require that our operators are monotone.
Starting from the empty set, apply $\Gamma_0$ until we obtain a fixed-point $X_0$ of $\Gamma_0$:
\[
    \emptyset, \Gamma_0(\emptyset), \Gamma_0(\Gamma_0(\emptyset)) \cup \Gamma_0(\emptyset), \dots
\]
Apply $\Gamma_1$ once, and then generate another fixed-point for $\Gamma_0$:
\[
    \Gamma_1(X_0), \Gamma_0(\Gamma_1(X_0)), \Gamma_0(\Gamma_0(\Gamma_1(X_0))) \cup \Gamma_0(\Gamma_1(X_0)), \dots
\]
Repeat this process until we obtain a fixed-point for both $\Gamma_0$ and $\Gamma_1$.
In case we are combining three operators $\Gamma_0,\Gamma_1,\Gamma_2$, we apply $\Gamma_2$ whenever we get a fixed-point for $\Gamma_0$ and $\Gamma_1$, and repeat until we get a fixed-point for all three operators.
The general case with $k$ operators follows the same idea.
\begin{lemma}
    \label{lem::from-induction-to-seq-of-models}
    Over $\mathsf{ACA}_0$, $\forall n.[\Sigma^1_1]^n$-$\mathsf{ID}$ proves $\forall n.\psi_1(2,n)$.
\end{lemma}
\noindent To finish the proof of Theorem \ref{thm::main-pi12-theorem} we use Proposition \ref{prop::ind-k_and_determinacy-forall-ver} and Corollary \ref{cor::comp_and_ref}.

For the proof of Lemma \ref{lem::from-induction-to-seq-of-models} we need only the existence of the fixed-points.
Therefore we use a simpler to state variation of $[\Sigma^1_1]$-$\mathsf{ID}$:
For $n\in \mathbb{N}$, $[\Sigma^1_1]^n$-$\mathsf{LFP}$ asserts that for any sequence $\tuple{\Gamma_i \mid i<n}$, there exists a smallest set $X$ which satisfies $\Gamma_i(X) = X$ for all $i<n$.
\begin{lemma}
    \label{lem::id-implies-lfp}
    $\forall n.[\Sigma^1_1]^n$-$\mathsf{ID}$ implies $\forall n.[\Sigma^1_1]^n$-$\mathsf{LFP}$ over $\mathsf{RCA}_0$.
\end{lemma}
\begin{proof}
    When obtaining the least simultaneous fixed-point of the operators $\Gamma_0, \dots, \Gamma_n$ by $\forall n.[\Sigma^1_1]^n$-$\mathsf{ID}$ we also register when each point enters the fixed-point.
    We only need to forget this information.
\end{proof}

\subsection{Warm-up}
As a warm up, we show that $[\Sigma^1_1]^3$-$\mathsf{LFP}$ implies $\psi_1(2,2)$, that is, we can use three $\Sigma^1_1$ operators $\Gamma_0,\Gamma_1,\Gamma_2$ to define coded models $M_0, M_1$ such that:
\[
    M_0 \subseteq_{\beta_2} M_1 \subseteq_\beta \groundmodel.
\]
The full version of Lemma \ref{lem::from-induction-to-seq-of-models} is on the next section.
As above, $\groundmodel$ is the fixed ground model.

A first {\em rough} idea of what $\Gamma_0$, $\Gamma_1$ and $\Gamma_2$ do is:
\begin{itemize}
    \item $\Gamma_0$ makes $M_0$ and $M_1$ $\beta$-submodels of the ground model.
    \item $\Gamma_1$ makes $M_0$ a $\beta_2$-submodel of $M_1$.
    \item $\Gamma_2$ puts $M_0$ inside $M_1$ as an element.
\end{itemize}
If we can define all of these operators, we are done.
The hardest operator to define correctly is $\Gamma_2$.
In order to do so, we will define auxiliary sets of `recipes' for making the sets in $M_0$ and $M_1$, and a copy $M_0^c$ of $M_0$ at convenient stages.
$M_0^c$ is a technical artifice used to guarantee $M_0\in M_1$.

We will have three kinds of recipes:
\begin{itemize}
    \item Recipes for applications of comprehension will have the form $\langle \mathtt{comp}, e, \bar j\rangle$ where $e$ is an index number and $\bar j$ are set parameters.
    \item Recipes where we copy an element of $M_1$ to $M_0$ will have the form $\langle \mathtt{subm}, e\rangle$ where $e$ is the index of some set in $M_1$ such that $M_1\models \forall Z.\theta((M_1)_e, Z)$, $M_0 \not\models \exists Y \forall Z \theta(Y,Z)$ and $\theta$ is some arithmetical formula.
    \item Recipes for putting $M_0^c$ inside $M_1$ are of the form $\langle \mathtt{elem}, e\rangle$ where $e$ is some index.
\end{itemize}
Each of $\mathtt{comp}, \mathtt{subm}$ and $\mathtt{elem}$ are arbitrarily choosen (pairwise different) natural numbers.
Each recipe will be the label for the set that it constructs, {\em i.e.}, the recipe $\rho$ instructs us how to define the set $(M_i)_\rho$.

We now describe the recipes one application of each operator constructs:
\begin{itemize}
    \item $\Gamma_0$: if $e\in \mathbb{N}$ and $\bar s$ is a finite sequence of elements of $M_i$, then $\langle \mathtt{comp}, e, \bar s\rangle$ is an recipe for $M_i$.
    \item $\Gamma_1$: if $\theta$ is an arithmetic formula with parameters $Y,Z$, $M_0 \not\models \exists Y \forall Z \theta(Y,Z)$ and $e$ is the least such that $M_1\models \forall Z \theta((M_1)_e,Z)$, then $\langle \mathtt{subm}, e\rangle$ is an recipe for $M_0$.
    \item $\Gamma_2$: if $e$ is the least such $\exists i \in (M_0)_e$ and $\neg \exists i\in (M_0^c)_e$, then $\langle \mathtt{elem}, e \rangle$ is an recipe for $M_1$.
\end{itemize}
These recipes will guarantee that $M_0$ and $M_1$ have the closures we want.

Now we describe how $\Gamma_0$ follows the recipes to create our sets:
\begin{itemize}
    \item If $\rho = \langle \mathtt{comp}, e, \bar j\rangle$ is a recipe for $M_i$, then $n\in (M_i)_\rho$ iff $\varphi(e, n, \overline{(M_i)_s}, X)$, where $\varphi$ is a universal lightface $\Sigma^1_1$-formula.
    \item If $\rho = \langle \mathtt{subm}, e\rangle$ is a recipe for $M_0$, then $n\in (M_0)_\rho$ iff $n\in (M_1)_e$.
    \item If $\rho = \langle \mathtt{elem}, e\rangle$ is a recipe for $M_1$, then $n\in (M_1)_\rho$ iff $n\in M_0^c$.
\end{itemize}
We will also require that the set made by each recipe is made only once.
Note that $\Gamma_0$ at the same time creates and follows recipes.
$\Gamma_0$ is a $\Sigma^1_1$-operator.

$\Gamma_1$ is a $\Sigma^1_1$-operator which adds new recipes for copying members of $M_1$ into $M_0$.
At last, $\Gamma_2$ is a $\Sigma^1_1$-operator which copies the current $M_0$ into the candidate $M_0^c$ and creates a recipe for copying the new $M_0^c$ into $M_1$.
This only adds elements to the old copy, so this is not problematic.

Let $X = (M_0, M_0^r, M_0^c, M_1, M_1^r)$, then:
\begin{itemize}
    \item if $X$ is a fixed-point of $\Gamma_0$, then $M_0\subseteq_\beta M_1 \subseteq_\beta \groundmodel$;
    \item if $X$ a fixed-point of $\Gamma_1$, then $M_0\subseteq_{\beta_2} M_1$; and
    \item if $X$ is a fixed-point of $\Gamma_2$, then $M_0\in M_1$.
\end{itemize}

\subsection{Multiple Induction}
In this section we show that $\forall n.[\Sigma^1_1]^n$-$\mathsf{LFP}$ implies $\forall n.\psi_1(2,n)$.
Fix $A\subseteq \mathbb{N}$ and $n\in \mathbb{N}$ such that $n\geq 1$.
We define a sequence of sets
\begin{align*}
     A \in \;&Y_0 \; \in  \;\;\cdots \; \in \;\;\; Y_n, \\
           &Y_0 \subseteq_{\beta_i}\cdots \subseteq_{\beta_i} Y_n \subseteq_{\beta_e} \groundmodel.
\end{align*}
using $2n-1$ $\Sigma^1_1$-inductive operators $\Gamma_0, \dots, \Gamma_{2n-2}$.

Write $\mathtt{comp}$, $\mathtt{subm}$ and $\mathtt{elem}$ for $0$, $1$ and $2$, respectively.
A tuple $\rho$ of natural numbers is a recipe iff there is a natural number $n$ and a natural number $j$ such that
\begin{itemize}
    \item $\rho = \tuple{\mathtt{comp}, e, j}$; or
    \item $\rho = \tuple{\mathtt{subm}, e}$; or
    \item $\rho = \tuple{\mathtt{elem}, e}$.
\end{itemize}
$\mathtt{comp}$ recipes are for closure under $\Pi^1_1$-comprehension, $\mathtt{subm}$ recipes are for making each model be a $\beta_2$-submodel of the next model, and $\mathsf{include}$ recipes are for putting a copy of each model into the next model.
From now on write $M_i$ for $3n$, $M_i^r$ for $3n+1$ and $M_i^c$ for $3n+2$.
Given a set $X$, we write $M^X_i$ for $(X)_{M_i}$.

The operator $\Gamma_0$ creates the recipes of the form $\tuple{\mathtt{comp}, e, j}$, with $j$ being the index of some non-empty set in the respective model.
$\Gamma_0$ also makes al the currently not-made-yet recipes.
As $\Gamma_0$ both creates and makes the recipes for closure under $\Pi^1_1$-comprehension, we can show that the models defined by a fixed-point of $\Gamma_0$ are coded $\beta$-models.
\begin{definition}
    $\Gamma_0$ is the $\Sigma^1_1$-inductive operator defined by:
    \begin{align*}
        x\in \Gamma_0(X) \iff
        [&x = \tuple{M_i^r, \tuple{\mathtt{comp}, e, s}}
             \land \forall j<\mathrm{lh}(s) \exists m.m\in (M^X_i)_{s_j}]\;
             \lor \\
        [&x = \tuple{M_i, \tuple{\tuple{\mathtt{comp}, e, s}, m}}
            \land \pi^1_1(e,m, (M^X_i)_{s}, A)
            \land \tuple{\mathtt{comp}, e, s}\in M^{r,X}_i]\;
            \lor \\
        [&x = \tuple{M_i, \tuple{\tuple{\mathtt{subm}, e}, m}}
            \land \tuple{e,m}\in M^X_{i+1}
            \land \tuple{\mathtt{subm}, e}\in M^{r,X}_i]\;
            \lor \\
        [&x = \tuple{M_{i+1}, \tuple{\tuple{\mathtt{elem}, e}, m}}
            \land \tuple{M_i^c,m}\in X
            \land \tuple{\mathtt{elem}, e}\in M^{r,X}_{i+1}\\&
            \land \neg \exists m.m\in (M^X_{i+1})_\tuple{\mathtt{elem}, e}].
    \end{align*}
    Here $\pi^1_1$ is a universal lightface $\Pi^1_1$ formula.
\end{definition}
\begin{lemma}
    \label{lem::gamma-0-is-ok}
    If $X$ is a fixed-point of $\Gamma_0$ and $i\in \mathbb{N}$, then $M_i^X$ is a coded $\beta$-model and $A\in M_i^X$.
\end{lemma}
\begin{proof}
    Suppose that $X$ is a fixed-point of $\Gamma_0$.
    Fix $i\in\mathbb{N}$ and $A\in M_i^X$.
    The hyperjump of $A$ is $\mathrm{HJ}(A) = \{\tuple{n,e} \mid \exists f.\pi^0_1(e,n,f,X)\}$, which is $\Pi^1_1$ relative to $A$ ($\pi^0_1$ is a universal lightface $\Pi^0_1$ formula).
    So we can define a recipe $\rho$ for $\mathrm{HJ}(A)$.
    Therefore $\rho\in M^{r,\Gamma_0(X)}_i$ and $\mathrm{HJ}(A)\in M^{\Gamma_0(\Gamma_0(X))}_i$.
    But $\Gamma_0(\Gamma_0(X)) = X$, so $\mathrm{HJ}(A)\in M^X_i$.
    Therefore $M^X_i$ is closed under hyperjumps, and thus is a coded $\beta$-model.

    As $\{a \mid a\in A\}$ is $\Pi^1_1$ with parameter $A$, we can similarly show that $A\in M^X_i$.
\end{proof}

$\Gamma_{2i+1}$ creates recipes to copy sets to $M_i^X$ from $M_{i+1}^X$, so that the former can become a $\beta_2$-submodel of the latter after one application of $\Gamma_0$.
If $M^X_i \subseteq_{\beta_2} M^X_{i+1}$, $\Gamma_{2i+1}$ does nothing.
\begin{definition}
    $\Gamma_{2i+1}$ is the $\Sigma^1_1$-inductive operator defined below:
    \begin{align*}
        x \in \Gamma_{2i+1}(X) \iff
        &x = \tuple{M_i^r, \tuple{\mathtt{subm},e}}  \\
        &\exists e_0 \exists s[e = \mu e.M_{i+1}^X\models \pi^1_1(e_0,(M_{i+1})_e^X,(M_i)^X_s)\\
        &\land M_i^X\not\models\exists Y\pi^1_1(e_0,Y, (M_i)^X_s)].
    \end{align*}
\end{definition}
\begin{lemma}
    \label{lem::gamma-2i+1-is-ok}
    Let $i\in\mathbb{N}$.
    If $X$ is a fixed-point of $\Gamma_{2i+1}$, then $M_i^X\subseteq_{\beta_2}M_{i+1}^X$.
\end{lemma}
\begin{proof}
    Let $X$ be a fixed-point of $\Gamma_{2i+1}$ and $\varphi$ be a $\Pi^1_2$ sentence with parameters in $M_i^X$.
    If $M_{i+1}^X\models \varphi$ then $M_i^X\models \varphi$, as otherwise $X$ would not be a fixed-point of $\Gamma_{2i+1}$.
\end{proof}

$\Gamma_{2i+2}$ checks if there is any difference between $M_i^X$ and $M_{i}^{c,X}$, and adds a new recipe for $M_i$ if that is the case.
$\Gamma_{2i+2}$ simultaneously copies $M_i^X$ over $M_i^{c,X}$.
After one application of $\Gamma_{2i+2}$ and one of $\Gamma_0$, we get that $M_i^X$ is an element of $M_{i+1}^{(\Gamma_0(\Gamma_{2i+2}(X)))}$.
\begin{definition}
    $\Gamma_{2i+2}$ is the $\Sigma^1_1$-inductive operator defined below:
    \begin{align*}
        x\in \Gamma_{2i+2}(X) \iff
        [&x = \tuple{M_i^r, \tuple{\mathtt{elem},e}} \\
        &\land \exists m.\tuple{e,m}\in M_{i}^X
        \land \tuple{e,m}\not \in M_{i}^{c, X} \\
        &\land \forall e'<e\forall m.\tuple{e',m}\in M_{i+1}^X \leftrightarrow \tuple{e',m} \in M_{i+1}^{c,X}]\;
        \lor \\
        [&x = \tuple{M_i^c, m} \land \tuple{M_i, m} \in X]
    \end{align*}
\end{definition}
\begin{lemma}
    \label{lem::gamma-2i+2-is-ok}
    For any fixed-point $X$ of $\Gamma_0$, $M_i^X\in M_{i+1}^{(\Gamma_0(\Gamma_{2i+2}(X)))}$.
\end{lemma}
\begin{proof}
    Fix $X$ and $i\in\mathbb{N}$.
    Then either $M^X_i = M^{c,X}_i$ or $M^X_i \neq M^{c,X}_i$.
    If $M^X_i = M^{c,X}_i$, there is a recipe $\tuple{include, e}\in M_{i+1}^{r,X}$ and as $X$ is a fixed-point of $\Gamma_0$, $M^{c,X}_i\in M^X_{i+1}$.
    If $M^X_i \neq M^{c,X}_i$, then there is $e$ such that $\tuple{\mathtt{elem},e}\in M_{i+1}^{r,\Gamma_{2i+2}(X)}\setminus M_{i+1}^{r,X}$.
    We also have $M^X_i = M^{c,\Gamma_{2i+2}(X)}_i$.
    Therefore $M^X_i \in M^{\Gamma_0(\Gamma_{2i+2}(X))}_{i+1}$.
\end{proof}

\begin{proof}[Proof of Lemma \ref{lem::from-induction-to-seq-of-models}]
    Fix $k\geq 1$.
    Suppose $\forall n.[\Sigma^1_1]^n$-$\mathsf{ID}$ holds.
    In particular, $[\Sigma^1_1]^{2k-1}$-$\mathsf{LFP}$ holds.
    Let $X$ be a sumultaneous fixed-point of the operators $\Gamma_0, \dots, \Gamma_{2k-2}$ defined above.
    By Lemmas \ref{lem::gamma-0-is-ok}, \ref{lem::gamma-2i+1-is-ok} and \ref{lem::gamma-2i+2-is-ok},
    \begin{align*}
     A \in \;&M^X_0 \; \in  \;\;\cdots \; \in \;\;\; M^X_n, \\
             &M^X_0 \subseteq_{\beta_i}\cdots \subseteq_{\beta_i} M^X_n \subseteq_{\beta_e} \groundmodel.
    \end{align*}
\end{proof}

\section{Determinacy of differences of $\Pi^0_3$ sets and $\Pi^1_3$-reflection for $\mathsf{Z}_2$}\label{section::the_z2_case}
In \cite{montalban2012limits}, Montalb{\'a}n and Shore showed that $\mathsf{Z_2}$ proves the determinacy of differences of (standard) finitely many $\Pi^0_3$ sets:
\begin{theorem}[\cite{montalban2012limits}*{Theorem 1.1}]
    For every $m\geq 1$, $\Pi^1_{m+2}$-$\mathsf{CA}_0$ proves $(\Pi^0_3)_m$-$\mathsf{Det}$.
\end{theorem}
\noindent A reversal is not possible, as MedSalem and Tanaka \cite{medsalem2007delta03} showed that $\Delta^1_1$-$\mathsf{Det}$ does not prove $\Delta^1_2$-$\mathsf{CA}_0$.
Monalb{\'a}n and Shore showed the following:
\begin{theorem}[\cite{montalban2014limits}*{Theorem 1.10.5}]
    Let $m\geq 1$ and $X\subseteq\mathbb{N}$, then $(\Pi^0_3)_{m}$-$\mathsf{Det}$ proves the existence of a $\beta$-model $\mathcal{M}$ of $\Delta^1_{m}$-$\mathsf{CA}_0$ with $X\in \mathcal{M}$.
\end{theorem}

While these two theorems have long and involved proofs, both can be formalized in a straightforward manner.
Denote by $\beta(\Delta^1_{m}\text{-}\mathsf{CA}_0)$ the sentence which states that for every $X$ there is a $\beta$-model $\mathcal{M}$ of $\Delta^1_{m}$-$\mathsf{CA}_0$ with $X\in \mathcal{M}$.
\begin{corollary}\label{cor::formalizing_montalban--shore}
    Formalizing in $\mathsf{ACA}_0$ the two theorems of Montalb{\'a}n and Shore above, we have:
    \begin{enumerate}
        \item $\mathsf{ACA}_0$ proves $\forall m.\bew{\Pi^1_{m+2}\text{-}\mathsf{CA}_0}{(\Pi^0_3)_m\text{-}\mathsf{Det}}$; and
        \item $\mathsf{ACA}_0$ proves $\forall m.\bew{(\Pi^0_3)_m\text{-}\mathsf{Det}}{\beta(\Delta^1_{m+2}\text{-}\mathsf{CA}_0)}$.
    \end{enumerate}
\end{corollary}
This corollary implies:
\begin{theorem}
    Over $\mathsf{ACA}_0$, $\forall n.(\Sigma^0_3)_n$-$\mathsf{Det}$ is equivalent to $\Pi^1_3$-$\refl(\mathsf{Z}_2)$.
\end{theorem}
\begin{proof}
    First suppose $\Pi^1_3$-$\refl(\mathsf{Z}_2)$.
    Let $m\in\mathbb{N}$.
    By Corollary \ref{cor::formalizing_montalban--shore}.1, $\bew{\Pi^1_{m+2}\text{-}\mathsf{CA}_0}{(\Pi^0_3)_m\text{-}\mathsf{Det}}$, so $\bew{\mathsf{Z}_2}{(\Pi^0_3)_m\text{-}\mathsf{Det}}$.
    By reflection, $(\Pi^0_3)_m\text{-}\mathsf{Det}$ holds.
    As this argument holds for any $m$, $\forall m.(\Pi^0_3)_m\text{-}\mathsf{Det}$ holds.

    Now, assume $\forall m.(\Pi^0_3)_m\text{-}\mathsf{Det}$.
    Suppose $\varphi$ is a $\Pi^1_3$ sentence such that $\bew{{Z}_2}{\varphi}$ and $\varphi$ is false.
    Let $m$ be such that $\bew{\Pi^1_{m}\text{-}\mathsf{CA}_0}{\varphi}$.
    We write $\neg\varphi$ as $\exists X \forall Y \exists Z.\theta$ with $\theta$ arithmetical.
    As $\varphi$ is false, there is $X_0$ such that $\forall Y \exists Z.\theta(X_0)$  holds.
    By Corollary \ref{cor::formalizing_montalban--shore}.2, there is a $\beta$-model $\mathcal{M}$ of $\Pi^1_{m}\text{-}\mathsf{CA}_0$ such that $X\in\mathcal{M}$.
    As $\mathcal{M}$ is a model of $\Pi^1_m$-$\mathsf{CA}_0$ and $\bew{\Pi^1_m\text{-}\mathsf{CA}_0}{\varphi}$, $\mathcal{M}\models \varphi$.
    Also, $\forall Y\in \mathcal{M}\exists Z.\theta(X_0)$ is true, so $\mathcal{M}\models\neg\varphi$.
    This is a contradiction, and so $\varphi$ is true.
    We conclude that $\Pi^1_3$-$\refl(\mathsf{Z}_2)$ holds.
\end{proof}

For any $m\in\omega$, $\Pi^1_3$-$\refl(\mathsf{Z}_m)$ proves $\forall n.(\Pi^0_{m+1})_n\text{-}\mathrm{Det}$, by Theorem 1.1 of Hachtman \cite{hachtman2017calibrating}.
To prove reflection from determinacy, one possible approach is to generalize \cite{montalban2014limits}*{Theorem 1.10} to $m^\text{th}$-order arithmetic.
Along with definitions for subsystems of $\mathsf{Z}_m$, we would also need workable definitions for $\beta$-models of higher-order arithmetic.
\begin{question}
    Let $\mathsf{Z}_m$ be an axiomatization of $m^\text{th}$-order arithmetic.
    \begin{enumerate}
        \item Is $\forall n.(\Pi^0_{m+1})_n\text{-}\mathrm{Det}$ equivalent to $\Pi^1_3$-$\refl(\mathsf{Z}_m)$ for any $m\in \omega$?
        \item Is $\forall m.\forall n.(\Pi^0_{m+1})_n\text{-}\mathrm{Det}$ equivalent to $\Pi^1_3$-$\refl(\mathsf{Z}_\omega)$.
    \end{enumerate}
\end{question}


\end{document}